\documentclass[a4paper]{amsart}

\usepackage[T1]{fontenc}
\usepackage[utf8]{inputenc}
\usepackage{paralist}
\usepackage{microtype}
\usepackage{lmodern}
\usepackage{dsfont,mathrsfs}
\usepackage{amsmath,amsthm,amssymb}
\usepackage{url}

\newtheorem{theorem}{Theorem}[section]
\newtheorem{lemma}[theorem]{Lemma}
\newtheorem{prop}[theorem]{Proposition}

\newtheorem*{thm*}{Theorem}

\theoremstyle{definition}
\newtheorem{definition}[theorem]{Definition}

\theoremstyle{remark}
\newtheorem{exa}[theorem]{Example}
\newtheorem{rem}[theorem]{Remark}
\numberwithin{equation}{section}

\renewcommand{\phi}{\varphi}
\newcommand{\coloneqq}{\mathrel{\mathop:}=}

\newcommand{\Z}{\mathds{Z}}
\newcommand{\N}{\mathds{N}}
\newcommand{\C}{\mathds{C}}
\newcommand{\R}{\mathds{R}}
\newcommand{\Real}{\mathrm{Re}}

\newcommand{\eps}{\varepsilon}
\DeclareMathOperator{\fix}{Fix}
\newcommand{\applied}[2]{\langle #1,#2\rangle}
\DeclareMathOperator{\lin}{span}

\title[Irreducible semigroups of Harris operators]
{On the Peripheral Point Spectrum and the Asymptotic Behavior of Irreducible Semigroups of Harris Operators}
\author{Moritz Gerlach}
\address{University of Ulm\\Institute of Applied Analysis\\89069 Ulm\\Germany}
\email{moritz.gerlach@uni-ulm.de}
\date{}

\begin{document}

\begin{abstract}
Given a positive, irreducible and bounded $C_0$-semigroup on a Banach lattice
with order continuous norm, we prove that the peripheral point spectrum of its generator is trivial
whenever one of its operators dominates a non-trivial compact or kernel operator.
For a discrete semigroup, i.e.\ for powers of a single operator $T$, 
we show that the point spectrum of some power $T^k$ intersects the unit circle at most in $1$. 

As a consequence, we obtain a sufficient condition for
strong convergence of the $C_0$-semigroup and for a subsequence of the powers of $T$, respectively.
\end{abstract}
\maketitle
\section{Introduction}

E. B. Davies proved in \cite{davies2005} that $\sigma_p(A)\cap i\R \subseteq\{0\}$ if $A$ is the generator 
of a positive and contractive semigroup on $\ell^p$ for some $1\leq p < \infty$. This result was generalized by V. Keicher
in \cite{keicher2006} to positive, bounded and 
strongly continuous semigroups on atomic Banach lattices with order continuous norm.
A further generalization to $(w)$-solvable semigroups on super-atomic Banach lattices was found by M. Wolff \cite{wolff2008}.

Independently, G. Greiner observed in \cite[Thm.\ 3.2 and Kor.\ 3.11]{greiner1982} that $\sigma_p(A)\cap i \R \subseteq\{0\}$ 
for the generator $A$
of a positive, contractive and strongly continuous semigroup on $L^p(\Omega,\mu)$, $1\leq p <\infty$, 
if the semigroup contains a kernel operator. 
This covers in particular the atomic case because every regular operator on an order 
complete atomic Banach lattice is a kernel operator.
Greiner's result is also in the center of a survey article by W.\ Arendt \cite{arendt2008} that includes 
a generalization to Abel bounded semigroups.

Under the assumption that the semigroup is irreducible,
we generalize Greiner's result in the present article in two respects.
First, we consider general Banach lattices with order continuous norm instead of $L^p$ 
and secondly, we merely assume 
that one operator of the semigroup dominates a non-trivial compact operator.
Under suitable assumptions, the latter holds in particular if the
semigroups contains a Harris operator, i.e.\ one operator dominates a non-trivial kernel operator.
More precisely, we prove the following theorem in Section~\ref{sec:mainresults}.
\begin{thm*}
	Let $E$ be a Banach lattice with order continuous norm and let
	$\mathscr{T}=(T(t))_{t\in[0,\infty)}$ be a positive, irreducible,
	bounded and strongly continuous semigroup on $E$ with generator $A$. Assume that there exists $t_0>0$ such that 
	$T(t_0)\geq K > 0$ where $K$ is a compact or a kernel operator.
	Then $\sigma_p(A) \cap i\R \subseteq \{ 0\}$.
\end{thm*}
This result remains true for semigroups of Harris operators
which are merely Abel bounded, see Proposition~\ref{prop:trivialspectrumabel}.

We also consider discrete semigroups by which we mean a family
$\mathscr{T} = (T^n)_{n\in\N_0}$ for some bounded operator $T$. 
If $T$ is positive, irreducible and power bounded such that $T^m$ dominates a non-trivial compact
or kernel operator for some $m\in\N$, then 
\[ \sigma_p(T^n) \cap \Gamma \subseteq\{1\}\] for some $n\in\N$ where $\Gamma$ denotes the unit circle.
This result, presented as Theorem~\ref{thm:trivialspectrumdiscrete}, is optimal 
in the sense that, under these assumptions, one can neither expect $\sigma_p(T) \cap \Gamma\subseteq\{1\}$ for the point spectrum of $T$
nor $\sigma(T^n)\cap \Gamma\subseteq\{1\}$ for the spectrum of some operator $T^n$,
see Examples \ref{ex:matrix} and \ref{ex:fullspectrum}.

One important application of the above-mentioned results lies in the analysis of the 
asymptotic behavior of a semigroup.
Greiner deduced from his result that a positive and contractive strongly continuous semigroup on $L^p$ 
converges strongly to an equilibrium if it admits a strictly positive fixed point and contains a kernel operator
\cite[Kor.\ 3.11]{greiner1982}. See also \cite[Thm.\ 4.2]{arendt2008} for a proof in the irreducible case.
In Section~\ref{sec:convergence}, we generalize this to bounded and irreducible semigroups of Harris operators 
on Banach lattices with order continuous norm.

In addition, we obtain an analogous result for discrete semigroups in Theorem~\ref{thm:discreteconvergence}.
 If $T$ is a positive, power bounded and irreducible 
Harris operator with non-trivial fixed space, then there exists some $k\in\N$ such that
the subsequence $T^{nk}$ converges strongly as $n$ tends to infinity.
If $T$ is even strongly positive, i.e. $Tx$ is a quasi-interior point of $E_+$ for all $x>0$,
we show that the whole sequence $T^n$ converges strongly.

The main tools for our study of the asymptotic behavior are a so-called zero-two law by G. Greiner
\cite[Thm.\ 3.7]{greiner1982} and a result by D. Axmann \cite[Satz 3.5]{axmann1980} on disjointness of
powers of an irreducible operator.
In order to improve the accessibility of both results we present their proofs in the appendices.

\section{Notation and Tools}

Let us fix some notation and recall some well-known facts from the theory of 
Banach lattices and positive operators.
The reader is referred to \cite{meyer1991} for a more detailed introduction.

Throughout, $E$ denotes a (real) Banach lattice with order continuous norm. 
We write $E_+$ for the positive cone of $E$ and $x>0$ for an  element $x\in E_+$ different from zero.
For $y\in E_+$ we denote by 
	\[ E_y \coloneqq \{ x\in E : \vert x\vert \leq c y \text{ for some } c>0 \} = \bigcup_{c\geq 0}[-cy,cy] \]
the \emph{principal ideal} generated by $y$. 
If $E_y$ is dense in $E$, then $y$ is called a \emph{quasi-interior point of $E_+$}.
We say that a positive functional $x' \in E'$ is \emph{strictly positive} if
$\applied{x'}{\vert x \vert} >0$ for all $x\neq 0$.
By $\fix(T) \coloneqq \ker(I-T)$ we denote the \emph{fixed space} of a linear operator $T:E\to E$.

\subsection{Positive Discrete and Strongly Continuous Semigroups}
For a convenient simultaneous dealing with discrete and strongly continuous semigroups
we use the following definition.
\begin{definition}
	Let $R=\N_0$ or $R = [0,\infty)$. 
	A family $\mathscr{T}=(T(t))_{t\in R}$ of bounded linear operators on $E$
	is called a \emph{semigroup on $E$}
	if $T(t+s)=T(t)T(s)$ for all $t, s \in R$. 
	A semigroup $\mathscr{T}=(T(t))_{t\in [0,\infty)}$ is called
	\emph{strongly continuous} if the mappings $t\mapsto T(t)x$ are continuous for all $x\in E$.
	A semigroup $\mathscr{T}=(T(t))_{t\in \N_0}$ is said to be \emph{discrete}. 

	A semigroup $\mathscr{T}$ is called \emph{positive} if $T(t)$ is positive for all $t\in R$
	and it is called \emph{bounded} if $\sup_{t\in R} \Vert T(t) \Vert < \infty$.

	Given a semigroup $\mathscr{T}=(T(t))_{t\in R}$, 
	a subset $A\subseteq E$ is said to be \emph{$\mathscr{T}$-invariant} if $T(t)A\subseteq A$ for all $t\in R$.
	The semigroup $\mathscr{T}$ is called \emph{irreducible} if there are no
	closed $\mathscr{T}$-invariant ideals in $E$ beside $\{0\}$ and $E$.
	By
	\[ \fix(\mathscr{T}) \coloneqq \bigcap_{t\in R} \ker(I-T(t)) \]
	we denote the \emph{fixed space} of a semigroup $\mathscr{T}$.
\end{definition}

It will be convenient to call a single operator $T$ \emph{irreducible} (\emph{power bounded}) if
the discrete semigroup $(T^n)_{n\in\N_0}$ is irreducible (bounded).

By $E_\C$ we denote the complexification of $E$, 
which is called a complex Banach lattice (see \cite[Sec.\ 2.2]{meyer1991}),
and by $T_\C : E_\C \to E_\C$ the complexification of a linear operator $T$.
If $\mathscr{T}=(T(t))_{t\in[0,\infty)}$ is a strongly continuous semigroup on $E$ 
with generator $A$, then so is 
$\mathscr{T}_\C=(T(t)_\C)_{t\in[0,\infty)}$ with generator $A_\C$.

	We recall that a strongly continuous semigroup with generator $A$ is
	called \emph{Abel bounded} if $\Real \lambda \leq 0$ for all $\lambda \in \sigma(A_\C)$
	and $\sup_{\lambda>0} \Vert \lambda (\lambda-A_\C)^{-1}\Vert < \infty$.
	A discrete semigroup $\mathscr{T}=(T^n)_{n\in\N_0}$ is called \emph{Abel bounded}
	if $\vert \lambda\vert \leq 1$ for all $\lambda\in\sigma(T_\C)$ and 
	$\sup_{\lambda > 1} \Vert (\lambda-1)(\lambda-T_\C)^{-1}\Vert < \infty$.

\begin{lemma}
	\label{lem:fixedpointadjoint}
	Let $\mathscr{T}=(T(t))_{t\in R}$ be a positive and Abel bounded semigroup on $E$ 
	which is strongly continuous or discrete.  Assume that
	$T(t)z\geq z$ for some $z\in E$, $z>0$, and all $t\in R$. Then there exists $0<x'\in \fix(\mathscr{T}')$.

	If in addition $\mathscr{T}$ is irreducible, $x'$ is strictly positive and $z$ is a quasi-interior point of $E_+$
	with $\fix(\mathscr{T}) = \lin\{z\}$.
\end{lemma}
\begin{proof}
	The existence of $0<x' \in \fix(\mathscr{T}')$ follows from \cite[Lem.\ V 4.8]{schaefer1974} in the discrete case
	and from the proof of \cite[Prop.\ 4.3.6]{arendt2001} for a strongly continuous semigroup, respectively.

	Now, assume that $\mathscr{T}$ is irreducible. Since the absolute null ideal
	\[N(x') \coloneqq \{ x \in E : \applied{x'}{\vert x\vert} = 0 \}\neq E\]
	is closed and $\mathscr{T}$-invariant,
	it follows that $N(x')=\{0\}$, i.e.\ $x'$ is strictly positive.
	Thus, $\applied{x'}{T(t)z - z}=0$ for all $t\in R$ implies that $z\in\fix(\mathscr{T})$. 
	By the same argument, we observe that
	$T(t)\vert x\vert=\vert x\vert$ for all $t\in R$ and $x\in \fix(\mathscr{T})$, 
	i.e.\ $\fix(\mathscr{T})$ is a sublattice of $E$.

	Let $x\in \fix(\mathscr{T})$.
	Then $x^+\coloneqq x\vee 0\in\fix(\mathscr{T})$ and $x^-\coloneqq x-x^+ \in \fix(\mathscr{T})$ and the principal 
	ideals generated by $x^+$ and $x^-$ are $\mathscr{T}$-invariant
	and disjoint. The irreducibility of $\mathscr{T}$ implies that either $x^+$ is a
	quasi-interior point of $E_+$ and $x^-=0$ or vice versa.
	Consequently, $\fix(\mathscr{T})$ is totally ordered and hence one-dimensional by \cite[Prop.\ II 3.4]{schaefer1974},
	i.e.\  $\fix(\mathscr{T}) = \lin\{z\}$.
\end{proof}

Next, we note a consequence of the well-known mean ergodic theorem.

\begin{lemma}
	\label{lem:meanergodic}
	Let $T:E\to E$ be a positive and Abel bounded linear operator.
	If there exists a quasi-interior point $e\in \fix(T)$ of $E_+$, then $T$ is mean ergodic, i.e.\
	\[E = \fix(T) \oplus \overline{(I-T)E}.\]
\end{lemma}
\begin{proof}
	For $n\in\N$ and $x\in E$ let $C(n)x \coloneqq \frac{1}{n} \sum_{k=0}^{n-1} T^k x$
	denote the Ces\`aro averages of $T$.
	It is well-known that the Abel boundedness of $T$ implies that 
	$M\coloneqq \sup_{n\in\N} \lVert C(n) \rVert < \infty$, see \cite[1.5]{emilion1985}.
	Hence,
	\[ \biggl\lVert \frac{T^n}{n} \biggr\rVert \leq 2 \biggl\lVert \frac{T^n}{n+1} \biggr\rVert 
	\leq 2 \lVert C(n+1)\rVert \leq 2M  \]
	for all $n\in\N$.
	Now, it follows from $\lim \frac{1}{n} T^n x = 0$ for all $x\in E_e$ that $\lim \frac{1}{n} T^n x = 0$ for all $x\in E$.
	
	Due to the order continuity of the norm on $E$, for every $c>0$
	the $C(n)$-invariant order interval $[-ce,ce]$ is $\sigma(E,E')$-compact,
	see \cite[Thm.\ 2.4.2]{meyer1991}. Thus, 
	the sequence $(C(n)x)$ has a weak cluster point for every $x\in E_e$ and 
	the mean ergodic theorem \cite[\S2 Thm.\ 1.1]{krengel1985} yields that
	$\lim C(n)x$ exists for all $x\in E_e$.  Consequently, by the uniform boundedness of the operators
	$C(n)$, $\lim C(n)x$ exists for all $x\in E$.
	The desired ergodic decomposition $E= \fix(T) \oplus \overline{(I-T)E}$ now follows from 
	\cite[\S2 Thm.\ 1.3]{krengel1985}.
\end{proof}

We will use the following version of the famous splitting theorem by Jacobs, de Leeuw and Glicksberg.
Let $\mathscr{L}_\sigma(E)$ denote the space of all bounded linear operator on $E$ endowed with the weak operator
topology.

\begin{theorem}[Jacobs -- de Leeuw -- Glicksberg]
	\label{thm:JdLG}
	Let $\mathscr{T} = (T(t))_{t\in R}$ be a positive and bounded semigroup on $E$ and
	denote by $\mathscr{S}$ the closure of $\{ T(t) : t \in R \}$ in $\mathscr{L}_\sigma(E)$.
	If there exists a quasi-interior point $e\in \fix(\mathscr{T})$ of $E_+$,
	then $\mathscr{S}$ contains a positive projection $Q$ commuting with every
	operator in $\mathscr{S}$ such that the closed subspace $F\coloneqq QE$ includes $\fix(\mathscr{T})$
	and $Q\mathscr{S} \coloneqq \{ QT : T\in\mathscr{S} \} \subseteq \mathscr{S}$ is a norm bounded group of positive operators
	with neutral element $Q$.

	If in addition $\mathscr{T}$ is irreducible, $Q$ is strictly positive in the sense that $Qy >0$ 
	for all $y>0$.
\end{theorem}
\begin{proof}
	It follows from the uniform boundedness principle 
	that the closure of a norm bounded set in $\mathscr{L}_\sigma(E)$ is bounded in norm, again.
	Moreover, the operators of $\mathscr{S}$ are positive and they commute since $\mathscr{T}$ is Abelian.

	Now, let $x\in E_e$, i.e.\ $x\in [-ce,ce]$ for some $c>0$. Then $\{ T(t)x : t\in R\}$ is a subset of $[-ce,ce]$ and hence,
	by the order continuity of the norm \cite[Thm.\ 2.4.2]{meyer1991}, relatively weakly compact. 
	Since $E_e$ is dense in $E$, $\{ T(t) : t\in R\}$ is relatively compact in $\mathscr{L}_\sigma(E)$ \cite[Cor.\ A 5]{nagel2000}.
	Now, it follows from \cite[\S2 Thm 4.1]{krengel1985} that there exists a (positive) projection $Q\in\mathscr{S}$ 
	that commutes with every operator in $\mathscr{S}$ and $Q\mathscr{S}$ is a (bounded) group with neutral element $Q$.
	Since $Q$ is in the $\mathscr{L}_\sigma(E)$-closure of $\mathscr{T}$, one has that $\fix(\mathscr{T})\subseteq\fix(Q) \subseteq F$.

	If in addition $\mathscr{T}$ is irreducible, then $Q$ is strictly positive because the closed ideal
	\[ N(Q) \coloneqq \{ y \in E : Q \vert y\vert =0 \} \]
	is $\mathscr{T}$-invariant and distinct from $E$.
\end{proof}

\subsection{Atoms}

We recall that an  element $a\in E_+$ is said to be an \emph{atom (of $E$)} if the generated ideal 
\[E_a \coloneqq \{ x\in E : \vert x\vert \leq ca \text{ for some } c>0 \} \]
is one-dimensional. If $E$ contains no atoms, $E$ is said to be \emph{diffuse}.

\begin{rem}
\label{rem:atomstoatoms}
Let $a\in E_+$ be an atom and $T: E\to E$ be positive. Then clearly $TE_a \subseteq E_{Ta}$.
If in addition $T$ is invertible with a positive inverse, i.e.\ $T$ is a lattice isomorphism, 
then \[ E_{Ta} = TT^{-1} E_{Ta} \subseteq T E_{T^{-1}Ta} = TE_a = \lin\{Ta\}. \]
Hence, every lattice isomorphism on $E$ maps atoms to atoms.
\end{rem}

The lemma below is a more general formulation of \cite[Prop.\ 3.5]{keicher2006}, 
followed by an analogous assertion in the discrete case.

\begin{lemma}
	\label{lem:atomsfixed}
	Let $\mathscr{T}=(T(t))_{t\in \R}$ be a positive and bounded strongly continuous group
	on $E$, i.e.\ $\mathscr{T}$ is a uniformly bounded family of positive linear operators on $E$ such
	that $T(0)=I$, $T(t)T(s)=T(t+s)$ for all $t,s\in\R$, and the mappings $t\mapsto T(t)x$ are continuous 
	from $\R$ to $E$ for every $x\in E$.
	Then $\fix(\mathscr{T})$ contains every atom of $E$.
\end{lemma}
\begin{proof}
	Since $\mathscr{T}$ is a group, every operator $T(t)$ is a lattice isomorphism and hence
	maps atoms to atoms by Remark \ref{rem:atomstoatoms}.
	Fix an atom $a\in E_+$. Then for every  disjoint atom $b \in E_+$ one has that
	$\lvert a-b \rvert =  a+ b$ and hence $\lVert a-b \rVert \geq \lVert a\rVert$.
	Thus, it follows from $\lim_{t\to 0} T(t)a  = a$ that $T(t)a \in E_a$ for all $t\in \R$.
	Therefore, $T(t)a =\exp(\lambda t)a$ for a $\lambda\in \R$ and all $t\in \R$ 
	by the semigroup law.
	Since the group $\mathscr{T}$ is bounded, we conclude that $\lambda = 0$ and hence $a\in \fix(\mathscr{T})$.
\end{proof}

\begin{lemma}
	\label{lem:atomsfixeddiscrete}
	Suppose that $E$ has atoms and let $T$ be an irreducible lattice isomorphism on $E$ such that
	\begin{align}
		\sup \{ \Vert T^{k} \Vert : k\in \Z \} < \infty. \label{eqn:sgrbdd}
	\end{align}
	Then there exists $n\in\N$ such that $T^n=I$.
\end{lemma}
\begin{proof}
	Let $a\in E_+$ be an atom.
	The ideal
	\[ J \coloneqq \{ x\in E : \vert x\vert \leq c (a+ Ta + \dots + T^m a) \text{ for some } c\geq 0,\, m\in\N\}\]
	is $T$-invariant and hence $\overline{J}=E$. 
	Since the norm on $E$ is order continuous, $\overline{J}$ equals $J^{\bot\bot}$, the band
	generated by $J$. Thus, for every $x>0$ 
	we find some $n\in\N_0$ such that $x\wedge T^n a >0$.
	In particular, $T^{-1}a\wedge T^{n-1} a >0$ for some $n\in\N$. As $T^{-1}$ maps atoms to atoms by Remark \ref{rem:atomstoatoms},
	this implies that $T^{n-1} a = cT^{-1}a$ for some $c>0$ and 
	$T^n a  = ca$. Now, it follows from assumption \eqref{eqn:sgrbdd} that $c=1$. Therefore, $T^ka \in \fix(T^n)$ for all $k\in\N_0$.
	Since every $T^k a$ is an atom, we conclude that
	\[ E= \overline{J} = \lin\{a,Ta,\dots,T^{n-1}a\} \subseteq \fix(T^n), \]
	which completes the proof.
\end{proof}

\subsection{Kernel and Harris Operators}
First, we recall 
that a linear operator $T:E\to E$ is called \emph{regular} if it is the difference of two positive operators.
Every regular operator is bounded \cite[Prop.\ 1.3.5]{meyer1991} and admits a modulus. Moreover, 
since $E$ is order complete, the regular operators on $E$ form a Banach lattice 
with respect to the regular norm $\lVert T\rVert_r \coloneqq \lVert \lvert T \rVert \rVert$ \cite[Prop.\ 1.3.6]{meyer1991}.

	We denote by $E'\otimes E$ the space of all finite rank operators on $E$.
	The elements of $(E'\otimes E)^{\bot \bot}$, the band generated by $E'\otimes E$
	in the space of all regular operators, are called \emph{(regular) kernel operators}. 
	A regular operator $T$ on $E$ is called a \emph{Harris operator} 
	if $T^n \not\in (E'\otimes E)^{\bot}$ for some $n\in\N$.

Note that a positive operator is a Harris operator if and only if some power 
dominates a non-trivial kernel operator. 
In the case where $E=L^p$, the band $(E'\otimes E)^{\bot\bot}$ consists precisely of those operators given
by a measurable kernel (see \cite[Prop.\ IV 9.8]{schaefer1974} or \cite[Thm.\ 3.3.7]{meyer1991}).

The following well-known lemma states that the kernel operators form an algebra ideal in the regular operators.

\begin{lemma}
	\label{lem:kernelprop}
	Let $K$ be a kernel operator. Then $KT \in (E'\otimes E)^{\bot\bot}$ 
		and $TK \in (E'\otimes E)^{\bot\bot}$ for every regular operator $T$.
\end{lemma}
\begin{proof}
We may assume that $K, T \geq 0$. Denote by $J$ the (lattice) ideal of regular operators 
		generated by $E'\otimes E$ and let $\mathscr{A}\coloneqq [0,K]\cap J$.
		Then $K = \sup \mathscr{A}$ by \cite[Prop.\ 1.2.6]{meyer1991} and 
		\[ Kx = \sup \{ Ax : A \in \mathscr{A} \} \quad (x\in E_+), \]
		since $\mathscr{A}$ is upwards directed.
		Now, consider the directed set $\mathscr{A}T \coloneqq \{ AT : A\in\mathscr{A} \}$. As every operator of $\mathscr{A}$ is
		dominated by a finite rank operator, so is every $AT$, i.e.\ $\mathscr{A}T \subseteq J$. Now, it follows from
		\[ KTx = \sup\{ ATx : A\in \mathscr{A} \} \quad (x\in E_+) \]
		that $KT=\sup \mathscr{A}T$ and hence $KT\in (E'\otimes E)^{\bot \bot}$.
		By a similar argument, using in addition the fact that $T$ is order continuous because of the order continuity of
		the norm, one observes that $TK \in (E'\otimes E)^{\bot\bot}$.
\end{proof}

The following lemma shows that
under certain conditions some power of a Harris operator always dominates a non-trivial compact operator.
\begin{lemma}
	\label{lem:harriscompact}
	Let $T:E \to E$ be a positive and power bounded Harris operator such that
	$\fix(T)$ contains a quasi-interior point of $E_+$ 
	and that there exists a strictly positive element in $\fix(T')$.
	Then $T^n \geq C >0$ for some $n\in\N$ and some compact operator $C$.
\end{lemma}
\begin{proof}
	After replacing $T$ with a suitable power $T^m$,
	we may assume that $T \geq K >0$ for some kernel operator $K:E\to E$.
	Let $S\coloneqq T-K$.
	Since the kernel operators form an algebra ideal, see Proposition \ref{lem:kernelprop},
	$K_n \coloneqq T^n - S^n \in (E'\otimes E)^{\bot\bot}$ for all $n\in\N$.

	Let $e\in \fix(T)$ be a quasi-interior point of $E_+$ and $x'\in \fix(T')$
	be strictly positive. It follows from $e = Te = Ke + Se > Se$ that $S^{n+1}e \leq S^ne$ 
	for all $n\in\N$.
	By the order continuity of the norm,
	$u\coloneqq \lim_{n\to\infty} S^n e \in \fix(S)$ exists and satisfies $0\leq u < e$.
	Since $x'$ is strictly positive, we infer from $Tu=Ku+Su \geq u$ and $\applied{x'}{Tu-u}=0$ that $u\in\fix(T)$.
	Hence, $v\coloneqq e-u >0$ is a fixed point of $T$ and $\lim S^n v =0$.
	As the operators $K_n$ are uniformly bounded, it follows from
	\[ K_n v = T^nv - S^n v  \to v \quad (n\to\infty) \]
	that
	$\lim K_n^2 v = v$. In particular, $K_m^2 \neq 0$ for a suitable $m\in\N$ large enough.

	It follows from \cite[Prop.\ 1.2.6]{meyer1991} that
	$K_m = \sup ([0,K_m]\cap J)$ where $J$ is the ideal generated by $E'\otimes E$. Thus, there exists a (bounded)
	sequence $(C_k) \subseteq [0,K_m]\cap J$ such that $\lim C_k x  =  K_m x$ for $x\in \{v,K_m v\}$.
	Since every $C_k$ is dominated by a
	finite rank operator, $C_k^2$ is compact for every $k\in\N$ by \cite[Cor.\ 3.7.15]{meyer1991}.
	We conclude from 
	\[ \Vert C_k^2 v - K_m^2 v \Vert \leq \Vert C_k (C_k v - K_m v) \Vert + \Vert (C_k - K_m) K_m v\Vert \]
	that $\lim C_k^2 v  = K_m^2 v \neq 0$. Therefore,
	$ 0 < C_k^2 \leq K_m^2 \leq T^{2m}$ for a suitable $k\in \N$.
\end{proof}

Arendt proved in \cite[Kor.\ 1.26]{arendt1979} that compact operators are disjoint from lattice isomorphisms on a diffuse and order complete
Banach lattice. This is an essential tool for our spectral analysis in Section \ref{sec:mainresults}.
In order to be more self-contained, we give a proof of this result for a Banach lattice with order continuous norm.

\begin{theorem}
\label{thm:VKdisjoint}
Let $V,K:E\to E$ be positive operators on a diffuse Banach lattice $E$ with order continuous norm. If
$V$ is a lattice isomorphism and $K$ is compact, then $V\wedge K =0$.
\end{theorem}
\begin{proof}
	First, we show that $S\coloneqq I\wedge K=0$. Aiming for a contradiction,
	we assume that $Sx>0$ for some $x\in E_+$.
	Since $E$ is Archimedean, we find some $c>0$
	such that $w\coloneqq cSx - x$ is not negative, i.e.\ it has a non-trivial positive part.
	Denote by $P$ the band projection onto $\{w^+\}^{\bot\bot}$, the band generated by $w^+$.
	Then
	\[ 0 < w^+ = Pw = P(cSx-x)=cPSx - Px.\]
	It follows easily from $S\leq I$ that $S$ is an orthomorphism, i.e.\ it commutes with every band projection.
	Thus, $SPx>0$ and hence $Px>0$.
	Since $E$ is diffuse, we are able to construct a sequence $(x_n)\subseteq E$ of pairwise disjoint elements, each satisfying $0<x_n<Px$,
	by applying \cite[Lem.\ 2.7.12]{meyer1991} inductively.
	Let  $P_n$ be the band projection onto $\{x_n\}^{\bot\bot}$ and define $u_n \coloneqq P_n x / \lVert P_n x \rVert$.
	As the orthomorphism $S$ is dominated by the compact operator $K$, $S$ itself is compact by \cite[Thm.\ 16.21]{aliprantis1985}.
	Hence, after passing to a subsequence, $(Su_n)$ converges to some $y\in E$.
	By the order continuity of the norm, the disjoint and order bounded sequence $(P_n y)$ converges to zero \cite[Thm.\ 2.4.2]{meyer1991}.
	Now, it follows from
	\[ \lVert P_n x \rVert u_n =P_n x = P_n P x \leq c P_n SPx = c P_n S x = c \lVert P_n x \rVert S u_n \]
	for all $n\in\N$ that
	\[ P_n y = \lim_{m\to\infty} P_n S u_m \leq \lim_{m\to\infty} P_n S c S u_m  = c S P_n y\]
	and hence 
	\begin{align*}
		 \lVert u_n \rVert \leq c\lVert P_n Su_n \rVert \leq c\lVert P_n S u_n - P_n y\rVert + c\lVert P_n y \rVert
		\leq c\lVert Su_n - y\rVert + c^2 \lVert SP_n y\rVert.
	\end{align*}
	for all $n\in\N$. The right-hand side tends to zero, which is a contradiction to $\lVert u_n \rVert=1$.
	Hence, $S=I\wedge K = 0$.

	In order to complete the proof, we recall that the mapping $T\mapsto TV$ is a lattice homomorphism on 
	the lattice of all regular operators on $E$, see \cite[Thm.\ 7.4]{aliprantis1985}.
	Therefore, $V\wedge K = (I\wedge KV^{-1})V$ and hence vanishes by the first part of the proof since $KV^{-1}$ is compact.
\end{proof}

\section{Triviality of the Peripheral Point Spectrum}
\label{sec:mainresults}

In this section we give the proofs of our main results, which are 
inspired by \cite[Thm.\ 3.1]{arendt2008}. There, the semigroup is assumed to
contain a kernel operator, which is a less general hypothesis than ours.

Again, let $E$ be a Banach lattice with order continuous norm.

\begin{theorem}
	\label{thm:trivialspectrum}
	Let $\mathscr{T}=(T(t))_{t\in [0,\infty)}$ be a positive, irreducible and
	bounded strongly continuous semigroup with generator $A$. If there exists $t_0>0$ such that
	$T(t_0)\geq K>0$ where $K$ is a compact or a kernel operator,
	then $\sigma_p(A_\C) \cap i\R \subseteq \{ 0\}$.
\end{theorem}
\begin{proof}
	Let $\alpha \in \R$ such that $i\alpha \in \sigma_p(A_\C) \cap i\R$.
	Then $T(t)_\C z = \exp(i\alpha t)z$ for some $z\in E_\C\backslash\{0\}$.
	Since $0<\vert z \vert = \vert T(t)_\C z \vert \leq T(t)\vert z\vert$ for all $t\geq 0$, 
	by Lemma~\ref{lem:fixedpointadjoint} 
	there exists a strictly positive $\phi\in \fix(\mathscr{T}')$ and $\fix(\mathscr{T}) = \lin\{e\}$,
	where $e\coloneqq \vert z \vert$ is a quasi-interior point of $E_+$.
	In view of Lemma \ref{lem:harriscompact}, we may now assume that $T(t_0)\geq K >0$ for a compact operator $K$.

	Denote by $\mathscr{S}$ the closure of $\{T(t) : t\geq 0\}$ in $\mathscr{L}_\sigma(E)$, 
	the bounded linear operators on $E$ endowed with the weak operator topology.
	Then, by Theorem~\ref{thm:JdLG}, $\mathscr{S}$ 
	contains a strictly positive projection $Q$ commuting with every operator in $\mathscr{S}$ such that
	$\fix(\mathscr{T})\subseteq F\coloneqq QE$
	and that $Q\mathscr{S}\coloneqq\{QS : S\in\mathscr{S}\}\subseteq\mathscr{S}$ is norm bounded 
	group of positive operators with neutral element $Q$.
	Hence, the restriction of the elements of $\mathscr{S}$ to $F$ defines a group of lattice isomorphisms.

	If $x\in F$, then $\vert x \vert  = \vert Q x \vert  \leq Q \vert x\vert$ and
	since $\phi$ is also a strictly positive fixed point of $Q'$, it follows that 
	$\applied{\phi}{Q\vert x\vert  - \vert x \vert}=0$ and hence $Q\vert x\vert  =\vert x \vert$.
	Thus, $F$ is a closed sublattice of $E$ and hence a Banach lattice. Moreover, the norm on $F$ is order continuous 
	because every monotone order bounded sequence in $F$ converges (see \cite[Thm.\ 2.4.2]{meyer1991}).

	Since $Q$ is strictly positive and the quasi-interior point $e$ belongs to $F$, $QK_{\mid F}e >0$.
	Thus, $T(t_0)_{\mid F} = QT(t_0)_{\mid F}$ dominates the non-trivial compact operator $QK_{\mid F}$.
	If $F$ were diffuse, compact operators would be disjoint from lattice isomorphisms 
	by Theorem \ref{thm:VKdisjoint}.
	Thus, there exists at least one atom $a\in F$. 
	Since $a\in\fix(\mathscr{T})=\lin\{e\}$ by Lemma~\ref{lem:atomsfixed}, we obtain that $e$ is an atom of $F$ and
	consequently,
	\[ F = \overline{F_e} = \lin\{e\} =\fix(\mathscr{T}).\]
	It remains to check that $z\in F_\C = Q_\C E_\C$. Then $z\in \fix(\mathscr{T}_\C)$ and therefore
	$\alpha=0$. 
	Let $w\in (E_\C)'$ such that $\lvert\applied{T(t)_\C z}{w}\rvert = \applied{z}{w} = \Vert z\Vert >0$.
	Since $Q_\C$ is in the closure 
	of $\mathscr{T}_\C$ in $\mathscr{L}_\sigma(E_\C)$, there exists $t\geq 0$ such that 
	\[ \vert \applied{T(t)_\C z - Q_\C z}{w}\vert < \Vert z\Vert.\]
	This implies that 
	\[ \vert \applied{Q_\C z}{w} \vert \geq  \lvert\applied{T(t)_\C z}{w} \rvert - \lvert \applied{T(t)_\C z - Q_\C z}{w}\rvert >0 \]
	 and hence $Q_\C z \neq 0$.
	Now, choose $w\in (E_\C)'$ vanishing on $z$ and let $\eps >0$.
	Again, there exists $t\geq 0$ such that
	\[ \vert\applied{T(t)_\C z - Q_\C z}{w}\vert < \eps\]
	which implies that
	\[ \vert \applied{Q_\C z}{w}\vert \leq \vert\applied{T(t)_\C z - Q_\C z}{w}\vert
	+\vert\applied{T(t)_\C z}{w}\vert < \eps.\]
	Since $\eps>0$ was arbitrary, $\applied{Q_\C z}{w}=0$. Thus, $Q_\C z \in \lin\{z\}$
 	and consequently $Q_\C z = z \in F_\C$ because $Q_\C$ is a projection.
\end{proof}

As we will see next, this result remains true for Abel bounded semigroups of Harris operators.

\begin{prop}
	\label{prop:trivialspectrumabel}
	Let $\mathscr{T}=(T(t))_{t\in [0,\infty)}$ be a positive, irreducible and
	Abel bounded strongly continuous semigroup with generator $A$. 
	If $T(t_0)$ is a Harris operator for some $t_0>0$, 
	then $\sigma_p(A_\C) \cap i\R \subseteq \{ 0\}$.
\end{prop}
\begin{proof}
	Assume that $T(t)_\C z = \exp(i\alpha t)z$ for some $z\in E_\C\backslash\{0\}$, $\alpha\in\R$ and all $t\geq 0$.
	Since $\lvert z\rvert \leq T(t)\lvert z\rvert$ for all $t\geq 0$, it follows from Lemma \ref{lem:fixedpointadjoint}
	that there exists a strictly positive linear form $\phi\in\fix(\mathscr{T}')$. 
	Now, we endow $E$ with the order continuous lattice norm
	\[ \Vert x \Vert_\phi \coloneqq \applied{\phi}{\vert x\vert} \quad (x\in E).\]
	Let $(F,\lVert \cdot \rVert_F)$ be the completion of $(E,\lVert \cdot \rVert_{\phi})$, i.e.\ the closure of $E$ in
	the bidual $(E,\Vert \cdot \Vert_\phi)''$.
	Then $\mathscr{T}$ uniquely extends to a positive strongly
	continuous contraction semigroup $\tilde{\mathscr{T}} = (\tilde T(t))_{t \in [0,\infty)}$ on $F$.

	We show that $\tilde{\mathscr{T}}$ is irreducible.
	Let $\tilde J$ be a closed and $\tilde{\mathscr{T}}$-invariant ideal in $(F,\lVert \cdot \rVert_F)$.
	Then $J\coloneqq \tilde J \cap E$ is a closed and $\mathscr{T}$-invariant ideal in $(E,\lVert \cdot \rVert)$ 
	and hence $J=\{0\}$ or $J=E$.
	Since $\phi$ is order continuous, $E$ is an ideal in $F$ by \cite[Lem.\ IV 9.3]{schaefer1974}.
	 This implies that $\tilde J$ is the closure of $J$ in $F$ and hence either $\tilde J=\{0\}$ or $\tilde J = F$.
	Thus, $\tilde{\mathscr{T}}$ is irreducible.

	Next, we verify that $\tilde T(t_0)$ is a Harris operator.
	To simplify notation, we assume that $T\coloneqq T(t_0) \not \in (E'\otimes E)^\bot$; 
	otherwise we replace $t_0$ with $nt_0$ for a suitable $n\in\N$.
	Hence, there exist $w'\in E'$, $w\in E$ and $z\in E_+$ such that $(T \wedge \lvert w' \otimes w\rvert) z > 0$.
	In view of Nakano's carrier theorem \cite[Thm.\ 1.4.11]{meyer1991}, the strictly positive $\phi$ is a weak order unit of $E'$.
	Since
	\[ \lvert w'\rvert \otimes \lvert w\rvert  
	= \sup_{m\in\N} ((m\phi \wedge \lvert w'\rvert ) \otimes \lvert w\rvert)
	= \sup_{m\in\N} ((m\phi \otimes \lvert w\rvert) \wedge (\lvert w'\rvert \otimes \lvert w\rvert)),\]
	one has that
	\begin{align*}
		0 < T\wedge \lvert w' \otimes  w\rvert  &= T\wedge (\lvert w'\rvert \otimes \lvert w\rvert) \\
		&= T \wedge \sup \{ (m\phi\otimes \lvert w\rvert) \wedge(\lvert w'\rvert \otimes \lvert w\rvert) : m\in\N \}\\
		&= \sup\{ T \wedge (m\phi \otimes \lvert w\rvert ) \wedge (\lvert w' \rvert \otimes \lvert w\rvert) : m\in\N\}.
	\end{align*}
	Hence, $m(T\wedge (\phi \otimes \lvert w\rvert)) \geq T\wedge (m\phi \otimes \lvert w\rvert)>0$ for some $m\in\N$.
	Since the extensions $\tilde \phi\otimes \lvert w\rvert$ and $\tilde T$ leave the ideal $E$ invariant,
	we conclude that 
	\begin{align*}
		(\tilde T\wedge (\tilde \phi \otimes \lvert w\rvert))z &= \inf_F \{ \tilde T(z-y) +\tilde\phi(y) \lvert w\rvert  : y\in F,\, 0\leq y\leq z\} \\
		&= \inf_E \{ T(z-y)+\phi(y)\lvert w\rvert : y\in E,\, 0\leq y\leq z\} \\
		&= (T\wedge (\phi\otimes \lvert w\rvert))z > 0. 
	\end{align*}
	This shows that $\tilde T = \tilde T(t_0) \not\in (F'\otimes F)^\bot$.
	
	Finally, we conclude from Theorem \ref{thm:trivialspectrum} that 
	$\sigma_p(A_\C)\cap i\R \subseteq \sigma_p(\tilde A_\C)\cap i\R \subseteq \{0\}$ where
	$\tilde A$ is the generator of $\tilde{\mathscr{T}}$.
\end{proof}

We continue with a consideration of discrete semigroups, i.e.\  powers of a single operator.
Let $\Gamma\coloneqq \{z\in \C : \vert z\vert=1\}$ denote the unit circle. 
The following simple example shows that one cannot expect that $\sigma_p(T_\C)\cap\Gamma \subseteq\{1\}$ 
in analogy to Theorem~\ref{thm:trivialspectrum}.

\begin{exa}
\label{ex:matrix}
On $E=\R^2$, consider the matrix
\begin{align}  T\coloneqq \begin{pmatrix} 0&1\\1&0 \end{pmatrix} \end{align}
Then $T$ is a positive, irreducible, contractive and compact operator but $-1 \in \sigma_p(T_\C)$.
\end{exa}

However, the following theorem states that the point spectrum of some power of $T_\C$ is trivial.

\begin{theorem}
	\label{thm:trivialspectrumdiscrete}
	Let $T:E\to E$ be a positive, irreducible and
	power bounded operator. If there exists $m\in\N$ 
	such that $T^m \geq K>0$ for a compact or a kernel operator $K$,
	then $\sigma_p(T_\C^n) \cap \Gamma \subseteq \{1\}$ for some $n\in\N$.
\end{theorem}
\begin{proof}
	If $\sigma_p(T_\C)\cap \Gamma = \emptyset$, then the assertion follows for $n=1$.
	Thus, we may assume that there exists $\alpha \in \R$ and $z\in E_\C$ such 
	that $z\neq 0$ and $T_\C z = \exp (i\alpha) z$.
	Since 
	\[ 0<\vert z \vert = \vert \exp(in\alpha) z \vert = \vert T_\C^n z \vert \leq T^n \vert z\vert \quad(n\in\N),\]
	Lemma~\ref{lem:fixedpointadjoint} applied to the discrete semigroup 
	$\mathscr{T}\coloneqq (T^n)_{n\in\N_0}$ yields
	that there exists a strictly positive $\phi\in \fix(T')$ and $\fix(T) = \lin\{e\}$, where
	$e\coloneqq \vert z \vert$ is a quasi-interior point of $E_+$.
	In view of Lemma \ref{lem:harriscompact}, we may now assume that $T^m \geq K >0$ for a compact operator $K$.

	As in the proof of Theorem~\ref{thm:trivialspectrum}, we conclude from 
	Theorem~\ref{thm:JdLG} that there exists a strictly positive projection 
	$Q:E\to E$ commuting with $T$ and having the following properties: Its range $F\coloneqq QE$ 
	is a closed sublattice of $E$ with order continuous norm and includes $\fix(T)$.
	Moreover, the restriction of $T$ to $F$ is a 
	lattice isomorphism with $\sup\{ \Vert T_{\mid F}^k \Vert : k \in \Z\} < \infty$.
	On the other hand, $T^m_{\mid F}$ dominates the non-trivial compact operator $QK_{\mid F}$
	and it follows from Theorem \ref{thm:VKdisjoint} that $F$ is not diffuse.

	Next, we show that $T_{\mid F}$ is irreducible.
	Let $J \subseteq F$ be a closed and $T$-invariant ideal with corresponding band projection $P: F \to J$
	and let $e' \coloneqq Pe$.
	Since $Te' \in J$, we observe that
	\[ Te' = TPe = PTPe \leq PTe = Pe = e'.\]
	Now, it follows from $\applied{\phi}{e'-Te'}=0$ that $e' \in \fix(T) = \lin\{e\}$ because $\phi$ is strictly positive.
	Hence, $e'=0$ or $e'=e$ which implies that either $J=\{0\}$ or $J=F$.

	Now, we conclude from Lemma~\ref{lem:atomsfixeddiscrete},
	applied to the restriction of $T$ to $F$, that there exists some $n\in\N$ such that $F\subseteq \fix(T^n)$.
	Let $\exp(i\beta) \in \sigma_p(T_\C^n)$ and denote by
	$\xi_1,\dots,\xi_n$ the $n$th roots of $\exp(i\beta)$, 
	i.e.\  $(\xi_k)^n = \exp(i\beta)$ for all $1\leq k\leq n$. Then we infer from
	\[ \exp(i\beta)-T_\C^n = (\xi_1 - T_\C)(\xi_2-T_\C) \dots (\xi_n -T_\C) \]
	that $\xi_k \in \sigma_p(T_\C)$ for at least one $1\leq k\leq n$.
	Pick $y\in E_\C\backslash\{0\}$ such that $T_\C y = \xi_k y$.
	Now, we observe that $y\in F_\C = Q_\C E_\C$ by the same arguments
	as in the proof of Theorem~\ref{thm:trivialspectrum}. 
	Indeed, since $Q_\C$ is in the closure of $\mathscr{T}_\C = (T_\C^k)_{k\in\N_0}$ in $\mathscr{L}_\sigma(E_\C)$,
	for $w \in (E_\C)'$ satisfying $\applied{w}{y} = \lVert y\rVert$ we find some $j\in \N_0$
	such that \[\lvert \applied{T_\C^j y - Q_\C y}{w} \rvert < \lVert y\rVert.\]
	As $\lvert \applied{T_\C^j y}{w} \rvert = \lvert \xi_k^j \applied{y}{w}\rvert = \lVert y\rVert$,
	it follows that $\lvert \applied{Q_\C y}{w}\rvert >0$ and hence $Q_\C y \neq 0$.
	On the other hand, for every $w\in (E_\C)'$ vanishing on $y$ and for every $\eps>0$ there exists
	$j\in \N_0$ such that
	$\lvert \applied{T_\C^j y - Q_\C y}{w} \rvert < \eps$
	and hence $\lvert \applied{Q_\C y}{w}\rvert < \eps$. Thus, $Q_\C y \in \lin\{y\}$ which implies that
	$Q_\C y = y \in F_\C$ since $Q_\C$ is a projection.
	
	Altogether, we proved that
	\[ y = T_\C^n y = \xi_k^n y = \exp(i\beta) y \]
	and hence $\exp(i\beta)=1$.
\end{proof}

\begin{rem}
\label{rem:trivialspectrum}
\begin{enumerate}[(a)]
	\item The assertion of Theorem \ref{thm:trivialspectrumdiscrete} remains true for every positive, irreducible and Abel bounded
	Harris operator $T:E\to E$. This can be proven analogously to Proposition \ref{prop:trivialspectrumabel}.

	\item \label{it:strictlypositive} It is well known that $\sigma_p(T_\C)\cap \Gamma \subseteq \{1\}$ for every power
	bounded  operator $T:E\to E$ which is \emph{strongly positive}, i.e.\ $Tx$ is a quasi-interior point of $E_+$ for all $x>0$,
	c.f.\ \cite[Prop.\ V 5.6]{schaefer1974}. 
	This can also be obtained from the proof of Theorem \ref{thm:trivialspectrumdiscrete}
	by observing that such an operator $T$ is irreducible and its restriction to $F$ is again
	strongly positive and therefore no lattice isomorphism unless $F$ contains atoms.

	\item The assumption $\mathscr{T}$ being irreducible cannot be dropped
	in Theorem \ref{thm:trivialspectrum} (and Theorem \ref{thm:trivialspectrumdiscrete}).
	Indeed, let $(T_1(t))_{t\in [0,\infty)}$ be a positive and bounded
	strongly continuous semigroup on $E$ such that $T_1(t_0) \geq K > 0$ for some compact operator $K$
	and some $t_0>0$.
	Now, let $(T_2(t))_{t\in [0,\infty)}$ be another positive and bounded  strongly continuous semigroup on $E$
	such that $i\alpha$ is an eigenvalue of its generator for some $\alpha\neq 0$.
	Then $T(t) \coloneqq T_1(t) \oplus T_2(t)$ defines a positive and bounded
	strongly continuous semigroup on $E\oplus E$ 
	where $T(t_0)$ dominates the compact operator $K\oplus 0$ but 
	$i\alpha$ is an eigenvalue of its generator.

	\item If $A$ is the generator of a bounded and positive 
	strongly continuous semigroup $\mathscr{T}=(T(t))_{t\in [0,\infty)}$ 
	such that $T(t_0)$ is compact for some $t_0>0$, it is easy to see that $\sigma(A_\C)\cap i\R \subseteq \{0\}$.
	Indeed, in this case $\mathscr{T}$ is eventually norm continuous
	\cite[Lem.\ II 4.22]{nagel2000} and hence $\sigma(A_\C)\cap i\R$ is bounded \cite[Thm.\ II 4.18]{nagel2000}.
	Since the boundary spectrum is cyclic by \cite[Thm.\ 2.4]{greiner1981},
	it follows that $\sigma(A_\C) \cap i\R \subseteq \{0\}$. 
	\end{enumerate}
\end{rem}

It is natural to ask if even $\sigma(T_\C^n)\cap \Gamma$ is trivial for some $n\in\N$ under the assumptions of
Theorem~\ref{thm:trivialspectrumdiscrete}. The following example shows that this is not the case.

\begin{exa}
\label{ex:fullspectrum}
	Fix $1<p<\infty$ and let $E\coloneqq \ell^p$.
	We construct a positive and irreducible contraction $T$ on $E$ such that $\sigma_p(T_\C) = \emptyset$
	and $\sigma(T_\C^n)\cap \Gamma = \Gamma$ for all $n\in\N$. 
	Since $E$ is atomic, $T$ is automatically a kernel operator and hence satisfies 
	the assumptions of Theorem~\ref{thm:trivialspectrumdiscrete}.

	Let $(b_n) \subseteq [0,1]$ be a decreasing sequence with $\prod_{n=1}^\infty (1-b_n) = \frac{1}{2}$
	and let $0 \leq a_n \leq 1-(1-b_n)^p$ be small enough such that $\sum_{n=1}^\infty a_n \leq 1$.
	For $x=(x_n) \in \ell^p$ we define
	\[ Tx \coloneqq \biggr( \sum_{n=1}^\infty a_n x_n,\, (1-b_1)x_1,\, (1-b_2)x_2,\, (1-b_3)x_3,\, \dots \biggr).\]
	If $x\geq 0$, then it follows from Jensen's inequality that
	\begin{align*}
		\Vert Tx \Vert^p &= \biggr(\sum_{n=1}^\infty a_nx_n \biggr)^p + \sum_{n=1}^\infty (1-b_n)^p x_n^p 
		\leq \sum_{n=1}^\infty a_n x_n^p + \sum_{n=1}^\infty (1-b_n)^p x_n^p \leq \Vert x\Vert^p
	\end{align*}
	which shows that $T$ is a contraction. 

	It is well-known that the complexification of $E$ equals $\ell^p(\C)$, the $p$-summable sequences in $\C$.
	Now, assume that there is $\lambda \in \sigma_p(T_\C)$ and denote by $z=(z_n) \in E_\C$, $z\neq 0$, 
	a corresponding eigenvector.
	Since $T_\C$ is injective and contractive, $0 < \vert\lambda\vert \leq 1$ and it follows by induction that
	\[ z_{n+1} = \frac{1}{\lambda^n} (1-b_1)(1-b_2) \dots (1-b_n)z_1 \]
	for all $n\in \N$. Hence, $z$ is not a null sequence. This is impossible and therefore
	$\sigma_p(T_\C)=\emptyset$.

	In order to calculate the peripheral spectrum of $T_\C$, we first point out
	that the discrete semigroup $(T_\C^n)_{n\in\N_0}$ is not strongly stable.
	Indeed, for $e_1 = (1,0,0,\dots)$ we have that 
	\[ \Vert T^n e_1 \Vert \geq \prod_{k=1}^n (1-b_k) \geq \frac{1}{2} \]
	for all $n\in\N$.
	Thus, it follows from the characterization of strong stability by Arendt, Batty, Lyubich and V\~{u}
	\cite[Thm.\ 5.1]{arendt1988} that $\sigma(T_\C)\cap \Gamma$ is uncountable. 
	Since $\sigma(T_\C)\cap \Gamma$ is cyclic \cite[Thm.\ V 4.9]{schaefer1974}, it is dense in $\Gamma$ and hence
	$\sigma(T_\C)\cap \Gamma = \Gamma$. By factorizing $\lambda - T_\C^n$ as in the proof of Theorem \ref{thm:trivialspectrumdiscrete}
	one observes that also $\sigma(T_\C^n)\cap \Gamma = \Gamma$ for all $n\in\N$.	
\end{exa}

\section{Strong Convergence of the Semigroup}
\label{sec:convergence}

The asymptotic behavior of a bounded strongly continuous semigroup $\mathscr{T}=(T(t))_{t\in [0,\infty)}$
on a complex Banach space $X$
is highly related to the peripheral (point) spectrum of its generator $A$.
If, for instance, $\mathscr{T}$ is known to be \emph{asymptotically almost periodic},
i.e.\ $X=X_0\oplus X_\mathrm{AP}$ where
\[ X_0 \coloneqq \{ x\in X : \lim_{t\to\infty} \Vert T(t)x \Vert  = 0\} \]
and
\[ X_{\mathrm{AP}} \coloneqq \overline{\lin} \{ x\in D(A) : Ax = i\alpha x \text{ for some } \alpha \in \R\},\]
then $\sigma_p(A)\cap i\R \subseteq \{0\}$ already implies that
$\lim_{t\to\infty} T(t)x$ exists for all $x\in X$.
This is the case if the generator $A$ has compact resolvent \cite[Prop.\ 5.4.7]{arendt2001}
or if the peripheral spectrum $\sigma(A)\cap i\R$ is countable and $\mathscr{T}$ is
\emph{totally ergodic} \cite[Thm.\ 5.5.5]{arendt2001}.

In the following, we prove a convergence result for semigroups of Harris operators 
with a non-trivial fixed space by adjusting techniques developed by Greiner to arbitrary operator semigroups
that might by discrete or strongly continuous.
Our main tool is Greiner's zero-two law, see Theorem \ref{thm:zerotwo} in the appendix.

Again, let $E$ be a Banach lattice with order continuous norm.
We start with the following proposition, a generalized version of \cite[Kor.\ 3.9]{greiner1982}, 
that yields sufficient conditions for strong convergence of a semigroup.

\begin{prop}
	\label{prop:convergence}
	Let $\mathscr{T} = (T(t))_{t\in R}$ be a positive, bounded and irreducible
	semigroup on $E$ which is strongly continuous or discrete
	such that $\fix(T(t))=\fix(\mathscr{T}) \not= \{0\}$ for all $t\in R$
	and assume that there are $r,s \in R$, $r>s$, such that $T(r)\wedge T(s) >0$.
	Then
	\[ \lim_{t\to\infty} T(t)x = \applied{x'}{x} e \quad (x\in E)\]
	for some strictly positive $x' \in \fix(\mathscr{T'})$ and a quasi-interior point $e \in \fix(\mathscr{T})$ of $E_+$.
\end{prop}
\begin{proof}
	 By Lemma~\ref{lem:fixedpointadjoint}, there exists a strictly positive $x'\in \fix(\mathscr{T'})$ and a 
	quasi-interior point $e$ of $E_+$ with $\fix(\mathscr{T}) = \lin\{e\}$.
	We may assume that $\applied{x'}{e} = 1$.
	By assumption, 
	\[ E_2 \coloneqq \{ y \in E : (T(t)\wedge T(t+\tau)) \vert y\vert = 0 \text{ for all } t\geq 0\} \neq E\]
	for $\tau \coloneqq r-s$.
	Since $\mathscr{T}$ is irreducible, it follows from Greiner's zero-two law,
	Theorem~\ref{thm:zerotwo}, that $E_2=\{0\}$ and
	\[ \lim_{t\to\infty} \vert T(t) - T(t+\tau)\vert e = 0. \]
	Hence,
	\[ \vert T(t)(I-T(\tau)) y \vert \leq \vert T(t)-T(t+\tau)\vert e \to 0 \quad (t\to\infty) \]
	for all $y\in [-e,e]$, i.e.\ $\lim T(t)z =0$ for all $z\in D\coloneqq (I-T(\tau))[-e,e]$. 
	As $D$ is total in $\overline{(I-T(\tau))E}$ and $\mathscr{T}$ is bounded, $\lim_{t\to\infty} T(t)x = 0$
	for all $x\in \overline{(I-T(\tau))E}$.
	By Lemma~\ref{lem:meanergodic}, $T(\tau)$ is mean ergodic, i.e.\
	\[ E = \fix(T(\tau)) \oplus \overline{(I-T(\tau))E}.\]
	Since $\fix(T(\tau))=\lin\{e\}$, the corresponding mean ergodic projection
	is given by $x'\otimes e$ which completes the proof.
\end{proof}

It seems to be rather technical to assume that
two operators $T(r)$ and $T(s)$ are disjoint.
However, by Theorem \ref{thm:axmann} due to Axmann, this holds if the semigroup 
contains an irreducible Harris operator. This leads us to the following theorems, 
our main results in this section.

\begin{theorem}
	\label{thm:convergencecontinuous}
	Let $\mathscr{T} = (T(t))_{t\in[0,\infty)}$ be a positive, bounded,
	irreducible and strongly continuous semigroup on $E$ with generator $A$ 
	such that $\fix(\mathscr{T}) \not= \{0\}$.
	If $T(t_0)$ is a Harris operator for some $t_0>0$, then
	\[ \lim_{t\to\infty} T(t)x = \applied{x'}{x} e \quad (x\in E)\]
	for some strictly positive $x' \in \fix(\mathscr{T'})$ and a quasi-interior point $e \in \fix(\mathscr{T})$ of $E_+$.
\end{theorem}
\begin{proof}
	By Lemma~\ref{lem:fixedpointadjoint}, $\fix(\mathscr{T})=\lin\{e\}$ for a quasi-interior point $e$ of $E_+$
	and there exists a strictly positive element $x' \in \fix(\mathscr{T'})$.
	Hence, Theorem~\ref{thm:trivialspectrum} implies that
	$\sigma_p(A_\C)\cap i\R=\{0\}$ and therefore
	\begin{align}
		\fix(T(t))=\fix(\mathscr{T}) = \lin\{e\} \quad (t>0) \label{eqn:equalfixedspaces}
	\end{align}
	(see \cite[Cor.\ IV 3.8]{nagel2000}).
	Next, we prove that $T\coloneqq T(t_0)$ is irreducible.
	Let $J\subseteq E$ be a $T$-invariant closed ideal and $P:E\to J$ the corresponding
	band projection. Then $TPe \in J$ and hence $TPe = PTPe \leq PTe = Pe$. 
	Since $\applied{x'}{Pe-TPe}=0$ and $x'$ is strictly positive, it follows that
	$Pe \in \fix(T) =\lin\{e\}$ which implies that $J=\{0\}$ or $J=E$.
	Therefore, it follows from Theorem~\ref{thm:axmann} that
	there exist some natural numbers $n<m$ such that $T^n\wedge T^m > 0$.
	Now, the assertion follows from Proposition \ref{prop:convergence}.
\end{proof}

In the discrete case, we obtain strong convergence of a subsequence $(T^{nk})_{k\in\N}$ for a fixed $n\in\N$
which is optimal in view of Example \ref{ex:matrix}.
If the operator is not only irreducible but even strongly positive, the sequence $(T^k)_{k\in\N}$ itself 
converges strongly to a projection of rank one.

\begin{theorem}
	\label{thm:discreteconvergence}
	Let $T:E\to E$ be a positive, power bounded
	and irreducible Harris operator with $\fix(T)\not=\{0\}$.
	Then there exists $n\in\N$ such that $T^{nk}$ converges strongly as $k$ tends to infinity.

	If $T$ is even strongly positive, i.e.\ $Tx$ is a quasi-interior point of $E_+$ for all $x>0$, 
then $\lim_{k\to\infty} T^k x = \applied{x'}{x} e$ for some strictly positive $x' \in \fix(T')$
and a quasi-interior point $e \in \fix(T)$ of $E_+$.
\end{theorem}
\begin{proof}
	By Theorem~\ref{thm:axmann}, there are natural numbers
	$a<b$ such that $T^a\wedge T^b > 0$. Let $n\coloneqq b-a$.
	Now, we conclude as in the proof of Proposition \ref{prop:convergence} that 
	\[ E = \fix(T^n) \oplus \overline{(I-T^n)E} \]
	and $\lim_{k\to\infty} T^kx=0$ for all $x\in\overline{(I-T^n)E}$.

	If $T$ is strongly positive, then $\sigma_p(T_\C)\cap \Gamma \subseteq \{1\}$ by Remark \ref{rem:trivialspectrum} 
	(\ref{it:strictlypositive}). Therefore, $\fix(T^k) = \fix(T)$ for all $k\in\N$ and 
	the assertion follows immediately from Proposition \ref{prop:convergence}.
\end{proof}

\begin{rem}
It is natural to ask whether Theorem \ref{thm:axmann}
holds true for positive and irreducible operators that merely dominate non-trivial compact operators,
which would allows us to generalize Theorem \ref{thm:convergencecontinuous} and \ref{thm:discreteconvergence}
accordingly.

This is not the case since there exists a positive, compact and irreducible 
operator $T$ on $L^2(\Gamma)$, the square-integrable functions on the unit circle endowed with
the Lebesgue measure, such that $T^n \wedge T^m = 0$ for all $n\neq m$.
In fact, based on a work of Varopoulos \cite{varopoulos1966}, Arendt constructed a self-adjoint and compact 
Markov operator $T$ on $L^2(\Gamma)$ such that $T^n \wedge T^m = 0$
whenever $n\neq m$ \cite[Ex.\ 3.7]{arendt1981}. By Theorem \ref{thm:irreduciblecomponent} 
below, there exists a $T$-invariant
band $B$ in $L^2(\Gamma)$ such that the restriction of $T$ to $B$ is irreducible and still compact.
Since $B$ is of the form $\{ f\in L^2(\Gamma) : f=0 \text{ on } A \}$
for some measurable $A\subseteq \Gamma$ by \cite[III \S1 Ex.2]{schaefer1974} and $\Gamma\backslash A$
is not a nullset, $B$ is in turn isomorphic to $L^2(\Gamma)$  by 
\cite[Cor.\ 6.6.7 and Thm.\ 9.2.2]{bogachev2007}.
\end{rem}

\begin{theorem}
\label{thm:irreduciblecomponent}
	Let $T:E\to E$ be a positive and compact operator. If there exist a quasi-interior point $e\in \fix(T)$ 
	of $E_+$ and a strictly positive $\phi \in \fix(T')$, then there are finitely many disjoint $T$-invariant bands
	$B_1,\dots,B_N \subseteq E$ distinct from $\{0\}$ such that 
	$E= B_1 \oplus \dots \oplus B_N$
	and the restriction of $T$ to $B_k$ is irreducible for all $k=1,\dots,N$.
\end{theorem}
\begin{proof}
	We assume that there is no $T$-invariant band $A\subseteq E$ except $\{0\}$
	such that the restriction of $T$ to $A$ is irreducible.
	Then, in particular, $T$ is not irreducible on $E$ and hence there is a $T$-invariant
	closed ideal $A_1$ distinct from $\{0\}$ and $E$. As the norm on $E$ is order continuous,
	every closed ideal is a projection band.
	Denote by $P_1: E \to A_1$ the corresponding band projection. Since 
	\[TP_1e = P_1TP_1e \leq P_1Te = P_1e\]
	and $\applied{P_1e - TP_1e}{\phi} = 0$, the strict positivity of $\phi$ implies that $TP_1e = P_1e$. 
	By linearity, $T(I-P_1)e = (I-P_1)e$ and thus
	$(I-P_1)E = A_1^\bot$ is a non-trivial $T$-invariant band, too. We may assume that $\applied{\phi}{e}=1$ and
	\[  \applied{P_1e}{\phi}\leq \frac{1}{2} \applied{e}{\phi};\]
	otherwise we replace $A_1$ with $A_1^\bot$ and $P_1$ with $I-P_1$.

	By our assumption, $T_{\mid A_1}$ is not irreducible and hence we find a $T$-invariant band $A_2 \subseteq A_1$ 
	distinct from $\{0\}$ and $A_1$ with band projection $P_2 : E \to A_2$ such that
	\[ \applied{P_2  e}{\phi} \leq \frac{1}{2} \applied{P_1 e}{\phi}.\]
	Inductively, we obtain a decreasing sequence $A_{n+1} \subseteq A_n$ of $T$-invariant bands with projections $P_n: E \to A_n$
	satisfying
	\[ \applied{P_n e }{\phi} \leq 2^{-n} \applied{e}{\phi} \quad (n\in\N).\]
	Since $\phi$ is strictly positive and the norm on $E$ is order continuous, this implies that $\lim P_ne =\inf P_n e = 0$.

	Now, consider the sequence $x_n \coloneqq P_n e/\Vert P_ne \Vert$. The compactness of $T$ yields that
	a subsequence of $(Tx_n) = (x_n)$ converges to some $x$ with $\Vert x\Vert = 1$.
	On the other hand, $x\in A_n$ for every $n\in\N$ because $x_k \in A_n$ whenever $k\geq n$. Thus, 
	$e = \lim(I-P_n)e \in \{x\}^\bot$ which implies that $x=0$, a contradiction.

	We proved the existence of a $T$-invariant band $\{0\} \neq B_1\subseteq E$ such that $T_{\mid B_1}$ is irreducible.
	If $B_1 \neq E$ we may apply the same argument to the restriction of $T$ to $B_1^\bot$ to obtain a $T$-invariant
	band $\{0\} \neq B_2 \subseteq B_1^\bot$ such that $T_{\mid B_2}$ is irreducible. 
	
	Continue this construction inductively as long as $B_1 \oplus \dots \oplus B_n \neq E$.
	Suppose that this process does not terminate after finitely many steps, i.e.\
	we obtain an infinite sequence $B_n$ of disjoint non-trivial bands such that $T_{\mid B_n}$ is irreducible. Denote by $Q_n : E \to B_n$
	the corresponding band projections and let $y_n \coloneqq Q_n e / \Vert Q_n e\Vert$.
	Then a subsequence of $T y_n = y_n$ converges to some $y\in E$ with $\Vert y\Vert = 1$ since $T$ is compact.
	On the other hand, $Q_k y = \lim_{n\to\infty} Q_k y_n = 0$ for every $k\in\N$ implies that
	$y\in B_k^\bot$ for all $k\in\N$. This shows that every $y_k$ is contained in $\{y\}^\bot$ and so is $y$. 
	Hence, $y=0$ contradicting $\Vert y\Vert =1$.

	We conclude that the process of constructing $B_1,B_2,\dots$ ends after finitely many steps, which completes the proof.
\end{proof}

\appendix

\section{Greiner's zero-two law}
\label{sec:zerotwo}

We present the proof of Greiner's zero-two law from \cite{greiner1982}
in a reformulation for Banach lattices with order continuous norm and without any 
continuity condition on the semigroup.

Throughout, let $\mathscr{T}=(T(t))_{t\in R}$ be a positive and bounded semigroup on $E$,
a Banach lattice with order continuous norm, and fix $\tau >0$.

\begin{theorem}[Greiner's zero-two law]
\label{thm:zerotwo}
	Assume that $\fix(\mathscr{T})$ contains a quasi-interior point $e$ of $E_+$
	and that there exists a strictly positive element in $\fix(\mathscr{T}')$.
	Then 
	\[ E_2 \coloneqq \{ y\in E : (T(t)\wedge T(t+\tau))\vert y\vert = 0 \text{ for all } t\in R \} \]
	and $E_0\coloneqq E_2^\bot$ are $\mathscr{T}$-invariant bands. Moreover, 
	if $P$ denotes the band projection onto $E_2$, then
	\[ \vert T(t) - T(t+\tau) \vert Pe = 2Pe \text{ for all }t\in R\]
	and
	\[ \lim_{t\to\infty} \vert T(t) - T(t+\tau) \vert (I-P)e =0.\]
\end{theorem}

To simplify notation, for $t \in R$ we define the positive operators 
	\begin{align*} S(t)\coloneqq T(t)\wedge T(t+\tau) \quad \text{and} \quad
D(t) \coloneqq \vert T(t) - T(t+\tau) \vert\end{align*} on $E$.
It follows immediately that
\begin{align}
S(t)x+\frac{1}{2}D(t)x = x \label{eqn:zerotwoSDa}
\end{align}
for all $x\in\fix(\mathscr{T})$ and $t\in R$. Further properties of $S(t)$ and $D(t)$
are provided by the following lemma.

\begin{lemma}
	\label{lem:zerotwoSD}
	Let $x' \in \fix(\mathscr{T}')$ be strictly positive and $\tau \in R$.
	Then the following assertions hold.
	\begin{enumerate}[(a)]
		\item \label{zerotwoSDb} $D(t)T(s) \geq D(t+s)$ and $T(s)D(t)\geq D(t+s)$ for all $t,s\in R$. 
			Moreover, if $x\in \fix(\mathscr{T})$, then $\lim_{t\to\infty} D(t)x \in \fix(\mathscr{T})$.
		\item\label{zerotwoSDc} $S(t)T(s) \leq S(t+s)$ and $T(s)S(t)\leq S(t+s)$ for all $t,s\in R$.
			Moreover, if $x\in \fix(\mathscr{T})$, then $\lim_{t\to\infty} S(t)x \in \fix(\mathscr{T})$.
		\item \label{zerotwoSDd} If $\lim_{t\to\infty} S(t)x > 0$ for all $0<x\in\fix(\mathscr{T})$, then 
		$\lim_{t\to\infty} S(t)^m x > 0$ for all $m\in\N$ and $x\in \fix(\mathscr{T})$, $x>0$.
	\end{enumerate}
\end{lemma}

\begin{proof}
		(\ref{zerotwoSDb}) For all $t, s\in R$ one observes that
		\[ D(t)T(s) = \vert T(t)-T(t+\tau)\vert\cdot \vert T(s)\vert \geq \vert (T(t)-T(t+\tau))T(s)\vert = D(t+s) \]
		and similarly that $T(s)D(t)\geq D(t+s)$.
		Let $x\in \fix(\mathscr{T})$, $x\geq 0$.  Then 
		\[D(t)x = D(t)T(s)x \geq D(t+s)x\]
		for all $t,s\in R$. Hence, by the order continuity of the norm,
		$y\coloneqq \lim D(t)x$ exists in $E$ and
		 \[ T(s)y = \lim_{t\to\infty} T(s)D(t)x \geq \lim_{t\to\infty} D(t+s)x = y \geq 0.\]
		 Finally, we conclude from 
		 $\applied{x'}{T(s)y - y} = 0$ that $y\in \fix(\mathscr{T})$ because $x'$ is strictly positive.
 	
		(\ref{zerotwoSDc}) For all $t,s\in R$ one observes that
		\begin{align*}
		  S(t)T(s) &= \frac{1}{2} (T(t+s)+T(t+\tau+s)-D(t)T(s) )\\ &\leq \frac{1}{2} (T(t+s)+T(t+\tau+s)-D(t+s)) = S(t+s)
		\end{align*}
		and similarly that $T(s)S(t) \leq S(t+s)$.
		Hence, $S(t)x$ is increasing for all positive $x\in\fix(\mathscr{T})$.
		Since $0 \leq S(t) \leq \frac{1}{2}(T(t)+T(t+\tau))$, we conclude as in the proof of part (\ref{zerotwoSDb}) that
		$\lim S(t)x$ exists in $\fix(\mathscr{T})$ for all $x\in \fix(\mathscr{T})$.

		(\ref{zerotwoSDd}) Let $x\in\fix(\mathscr{T})$, $x>0$, and define recursively 
		$x_k \coloneqq \lim_{t\to\infty} S(t)x_{k-1}$ for all $k\in\N$
		where $x_0 \coloneqq x$.
		Then $x_k\in\fix(\mathscr{T})$ by part (\ref{zerotwoSDc}) and $x_k>0$ by assumption.
			It follows by induction that
		\[ S(t)^m x-x_m = \sum_{j=1}^m S(t)^{m-j}( S(t)x_{j-1} -x_j )\]
		for all $m\in\N$ and $t\in R$.
		As $\Vert S(t) \Vert \leq \sup_{t\in R}\Vert T(t)\Vert $, we conclude that
		\[ \lim_{t\to\infty} S(t)^m x = x_m>0\]
		for all $m\in\N$.
\end{proof}

The key for the proof of the zero-two law is the following combinatorial lemma.

\begin{lemma}
	\label{lem:zerotwoestimation}
	For every $m\in \N$
	\[ 2^{-m}\biggr( \sum_{j=1}^{m} \biggr\vert \binom{m}{j}-\binom{m}{j-1} \biggr\vert +2\biggr) \leq \frac{2}{\sqrt{m}}. \]
\end{lemma}
\begin{proof}
	For $k\in\N$ and $m=2k-1$ we have
	\begin{align*}
		\sum_{j=1}^m \biggr\vert \binom{m}{j}-\binom{m}{j-1} \biggr\vert +2
		&= 2\sum_{j=1}^{k-1}\biggr(\binom{m}{j}-\binom{m}{j-1}\biggr)+2\\
		&= 2\binom{m}{k-1} = \binom{2k}{k} \\
		&= \sum_{j=1}^k \biggr(\binom{m+1}{j}-\binom{m+1}{j-1}\biggr)+1\\
		&= \frac{1}{2}\sum_{j=1}^{m+1}\biggr\vert \binom{m+1}{j}-\binom{m+1}{j-1} \biggr\vert+1.
	\end{align*}
	It follows from Stirling's formula that the central binomial coefficient can be estimated by
	\[ \binom{2k}{k} \leq \frac{2^{2k}}{\sqrt{2k}} \quad (k\in\N).\]
	Thus, we obtain that
	\begin{align*}
		2^{-m}\biggr( \sum_{j=1}^m \biggr\vert \binom{m}{j}-\binom{m}{j-1}\biggr\vert +2\biggr)
		&= 2^{-(m+1)} \biggr(\sum_{j=1}^{m+1}\biggr\vert \binom{m+1}{j}-\binom{m+1}{j-1}\biggr\vert +2\biggr)\\
		&= 2^{-m} \binom{2k}{k} 
		\leq \frac{2}{\sqrt{m+1}} \leq \frac{2}{\sqrt{m}},
	\end{align*}
	which completes the proof.
\end{proof}

\begin{proof}[Proof of Theorem \ref{thm:zerotwo}]
	By Lemma \ref{lem:zerotwoSD} (\ref{zerotwoSDc}), we have 
	\[ S(t)\vert T(s)y \vert \leq S(t)T(s)\vert y\vert \leq S(t+s)\vert y \vert = 0\]
	for all $y\in E_2$ and $t, s\in R$, which shows that $E_2$ is $\mathscr{T}$-invariant.
	Let $x'\in \fix(\mathscr{T}')$ be strictly positive. Since
	\[ 0 \leq T(t)Pe = PT(t)Pe \leq PT(t)e = Pe \]
	and $\applied{x'}{Pe - T(t)Pe}=0$ for all $t\in R$, it follows that $Pe \in \fix(\mathscr{T})$. 
	Define $e_0 \coloneqq (I-P)e \in E_0$ and $e_2 \coloneqq Pe\in E_2$.
	As $e_0 = e-e_2 \in \fix(\mathscr{T})$, we conclude that $E_0$, which equals 
	the closure of the principle ideal generated by $e_0$, is $\mathscr{T}$-invariant.

	It follows immediately from \eqref{eqn:zerotwoSDa} that $D(t)e_2 =2 e_2$ for all $t\in R$. 
	Hence, it remains to show that $\lim D(t) e_0 = 0$. 
	For simplicity, we omit the index $0$ and write $E=E_0$, $e=e_0$, $T(t)=T(t)|_{E_0}$ and so on.
	Now, $\lim S(t)y > 0$ for all $y\in \fix(\mathscr{T})$, $y>0$, by Lemma~\ref{lem:zerotwoSD} (\ref{zerotwoSDc})
	and the definition of $E_0$.

	Assume that $h\coloneqq \lim D(t)e >0$. As $h\leq 2e$, there exists
	$m\in\N$ such that \[ k \coloneqq \biggr( h -\frac{2}{\sqrt{m}} e\biggr)^+>0.\]
	Moreover, $k\in \fix(\mathscr{T})$ since the fixed space is a sublattice. Indeed,
	if $y\in\fix(\mathscr{T})$, it follows from $T(t)\vert y\vert \geq \vert T(t)y\vert = \vert y\vert$
	and $\applied{x'}{T(t)\vert y\vert - \vert y\vert} = 0$ for all $t\in R$ that
	$\vert y\vert \in\fix(\mathscr{T})$ because $x'$ is strictly positive.
	Now, Lemma \ref{lem:zerotwoSD} (\ref{zerotwoSDd}) yields that 
	$S(t_0)^m k > 0$ for some $t_0\in R$. Let $t_1 \coloneqq m(t_0 + \tau)$ and define the operator
	\begin{align}
	 U \coloneqq T(t_1) - 2^{-m} S(t_0)^m \biggr(I+T(\tau)\biggr)^m. \label{eqn:Tt1}
	\end{align}
	It follows from $S(t_0)(I+T(\tau)) \leq T(t_0+\tau)+T(t_0)T(\tau) = 2T(t_0+\tau)$ that
	$U$ is positive.
	Moreover,
	\begin{align}
	 T(n t_1) = U^n + R_n 2^{-m} \biggr( I + T(\tau)\biggr)^m \label{eqn:Tnt1}
	 \end{align}
	for all $n\in\N$ where $R_1 \coloneqq S(t_0)^m$ and $R_{n+1} \coloneqq U^nR_1 + R_nT(t_1)$.
	We infer from $e=U^ne+R_ne$ that $0\leq R_n e \leq e$ and $0\leq U^n e \leq e$ for all $n\in\N$.
	Now, by Lemma \ref{lem:zerotwoSD} (\ref{zerotwoSDb}) and Lemma \ref{lem:zerotwoestimation}, we obtain that
	\begin{align*}
		h \leq D(nt_1)e &= \vert T(nt_1)(I-T(\tau))\vert e \\
		&\leq 2U^n e + R_n2^{-m}
		\biggr\vert \sum_{j=0}^m \binom{m}{j} T^j(\tau)(I-T(\tau))\biggr\vert e \\
		&=2U^n e + R_n 2^{-m} \biggr\vert  \sum_{j=0}^m \binom{m}{j}T^j(\tau) 
		- \sum_{j=1}^{m+1}\binom{m}{j-1}T^j(\tau)\biggr\vert e \\
		&=2U^n e + R_n2^{-m} \biggr( 2e + \sum_{j=1}^{m} \biggr\vert \binom{m}{j}-\binom{m}{j-1}\biggr\vert e\biggr)\\
		&\leq 2 U^ne+R_n\frac{2}{\sqrt{m}}e  \leq 2 U^ne+\frac{2}{\sqrt{m}}e
	\end{align*}
	for every $n\in\N$. Let $y\coloneqq \lim_{n\to\infty} U^n e \geq 0$. Then
	$h\leq 2 y + \frac{2}{\sqrt{m}} e$ and hence
	\[0 < k = \biggr( h -\frac{2}{\sqrt{m}} e\biggr)^+ \leq 2y. \]
	Since $y$ is a fixed point of $U$, equation \eqref{eqn:Tnt1} yields that $T(n t_1)y \geq y \geq 0$
	and we conclude from $\applied{x'}{T(nt_1)y - y}=0$ that $T(n t_1)y=y$ for every $n\in\N$.
	By equation \eqref{eqn:Tt1} we have \[ 0 = S(t_0)^m \biggr( I+T(\tau) \biggr)^m y \geq S(t_0)^my \geq 0. \]
	Therefore, $0 < S(t_0)^m k \leq 2S(t_0)^m y = 0$ which contradicts the preceded
	observation that $S(t_0)^mk>0$. 
	Hence, $h=\lim D(t)e = 0$.
\end{proof}

\section{Axmann's theorem}
\label{sec:axmann}

We give a proof Axmann's theorem from \cite[Satz 3.5]{axmann1980}
stating that not all powers of an irreducible Harris operator can be disjoint.
A proof of this for $E=L^p$ can be found in \cite[Sec.\ 6]{arendt2008}.

We start with a version for $L$-spaces and reduce the general case to it in what follows.
Let us recall that a Banach lattice $E$ is said to be a \emph{$L$-space} if 
\[ \lVert x + y \lVert = \lVert x\rVert + \lVert y\rVert\]  holds for all $x,y \in E_+$.

\begin{prop}
\label{prop:axmann}
Let $T$ be a positive and irreducible operator on a $L$-space $E$
such that $T\not\in (E'\otimes E)^\bot$.
Then there is $n\in \N$, $n\geq 2$, such that $T\wedge T^n >0$.
\end{prop}
\begin{proof}
	Aiming for a contradiction, we assume that $T\wedge T^n =0$ for all $n \geq 2$.
	Since $E$ is a $L$-space, we may identify $E'$ with $C(K)$ for some compact space $K$ 
	by Kakutani's theorem \cite[Thm.\ 2.1.3]{meyer1991}.
	For $n,m \in\N$, $n\geq 2$, define
	\[ A_n \coloneqq \{ T'h + T'^n(\mathds{1}-h) : h\in C(K),\, 0\leq h\leq \mathds{1}\} \subseteq C(K)\]
	and
	\[ O_{n,m} \coloneqq \{t \in K : h(t) < 1/m \text{ for some } h\in A_n\}.\]
	It follows from our assumption and Synnatschke's theorem \cite[Prop.\ 1.4.17]{meyer1991} that
	\[ \inf A_n = (T' \wedge T'^n)\mathds{1} = (T\wedge T^n)' \mathds{1} = 0\] for all $n\geq 2$.
	Now, we show that each of the open sets $O_{n,m}$ is dense in $K$.
	Assume the opposite. Then there exists a non-empty open
	set $U\subseteq K\backslash O_{n,m}$ for some $n,m \in \N$.
	By Urysohn's theorem, we can construct a continuous function $g:K\to [0,\frac{1}{m}]$ vanishing on $O_{n,m}$
	such that $g(t_0)=\frac{1}{m}$ for some $t_0 \in U$.
	Hence, $g>0$ is a lower bound of $A_n$, which is impossible.
	Therefore, every $O_{n,m}$ and, by Baire's theorem, also $G\coloneqq \cap_{n,m} O_{n,m}$ is dense in $K$.
	Note that for all $t\in G$ and $n\geq 2$ one has that 
	\[ \applied{T''\delta_t \wedge T''^n \delta_t}{\mathds{1}} = \inf\{ h(t) : h\in A_n \} = 0\]
	and consequently $T'' \delta_t \wedge T''^n \delta_t = 0$.

	By assumption, there are $x\in E_+$ and $y'\in E'_+$ such that $T$ is not disjoint from $R = y' \otimes x$.
	Then $R'= x\otimes y'$ corresponds to a rank-one operator 
	$\mu\otimes g$ on $C(K)$ for some $\mu \in C(K)'_+$ and $g\in C(K)_+$.
	Since $R'\wedge T' \geq (R\wedge T)' >0$ there exists some $e\in C(K)_+$, such that
	$\varrho \coloneqq (R'\wedge T')e>0$. For all $0\leq h\leq e$ and $t\in K$ it follows from
	\[ \varrho = (R'\wedge T')h + (R'\wedge T')(e-h) \leq R'h+T'(e-h) \]
	that
	\[ \varrho(t) \leq \applied{R'h}{\delta_t}+\applied{e-h}{T''\delta_t} = g(t)\applied{\mu}{h}+\applied{e-h}{T''\delta_t}.\]
	Taking the infimum over all $0\leq h\leq e$ yields that $\varrho(t) \leq (g(t)\mu \wedge T''\delta_t)e$
	for all $t\in K$.
	Now, fix $t\in \{ s\in K : \varrho(s)>0\} \cap G$ which exists since $G$ is dense in $K$.
	Then $\nu \coloneqq g(t)\mu \wedge T''\delta_t >0$ because $\varrho(t)>0$.
	As $\nu$ is dominated by $g(t)\mu$ and $\mu$ corresponds to $x \in E$, $\nu$ itself corresponds 
	to a vector $v\in E$, $v>0$, since $E$ is an ideal in $E''$. This vector $v$ satisfies
	\[ v \wedge T^n v \leq T''\delta_t \wedge T''^{n+1} \delta_t = 0\]
	for all $n\in \N$ because $t\in  G$.
	
	Finally, we consider $w\coloneqq Tv$. 
	If $w=0$, then the closed ideal $\overline{E_v}$ is $T$-invariant and non-trivial since $T\neq 0$.
	If $w>0$, then the closure of the $T$-invariant ideal
	\[ J\coloneqq  \{ z\in E : \lvert z\rvert \leq c(w+Tw+\dots+T^k w) \text{ for some } c>0 \text{ and } m\in\N\}\]
	is non-trivial because $w\in \overline{J}$ and $v \in J^\bot$.
	In both cases, $T$ cannot be irreducible.
	Thus, we conclude that $T\wedge T^n > 0$ for some $n\geq 2$.
\end{proof}

\begin{theorem}
\label{thm:axmann}
Let $T$ be a positive and irreducible operator on $E$, a Banach lattice with order continuous norm,
such that $T\not\in (E'\otimes E)^\bot$.
Then there is $n\in \N$, $n\geq 2$, such that $T\wedge T^n >0$.
\end{theorem}
\begin{proof}
	Fix $\lambda > \lVert T \rVert$ and $y'\in E'_+$, $y'\neq 0$. Then $z' \coloneqq (\lambda - T')^{-1}y'$ satisfies
	$\lambda z' - T'z' = y'$ and hence $T'z' \leq \lambda z'$.
	This implies that the closed ideal $\{ x\in E : \applied{\lvert x\rvert}{z'}=0\}$ is $T$-invariant
	and thus equal to $\{0\}$. 
	Therefore, $\lVert x\rVert_{z'} \coloneqq \applied{z'}{\lvert x\rvert}$ defines an order continuous lattice norm on $E$.
	Let $(F,\lVert \cdot \rVert_F)$ be the completion of $(E,\lVert \cdot\rVert_{z'})$, i.e. the closure of $E$ in 
	$(E,\lVert \cdot\rVert_{z'})''$. Then $F$ is an $L$-space and,
	since $\lVert Tx \rVert_F \leq \lambda \lVert T\rVert_F$ for all $x\in E$, $T$ uniquely extends to a positive operator $\tilde T$ on $F$.

	Now, it follows as in the proof of Proposition \ref{prop:trivialspectrumabel} that $\tilde T$ is irreducible and
	$\tilde T \not\in (F'\otimes F)^{\bot}$.

	Since the norm on $E$ is order continuous, so is $z'$ and hence $E$ is an ideal in $F$ by \cite[Lem.\ IV 9.3]{schaefer1974}.
	Thus, for $x\in E_+$ and $n\in\N$ we observe that
	\begin{align*}
		(\tilde T\wedge \tilde T^n)x &= \inf_F \{ \tilde T(x-y) +\tilde T^n y  : y\in F,\, 0\leq y\leq x\}\\
		&= \inf_E \{ T(x-y)+T^n y : y\in E,\, 0\leq y\leq x\} \\
		&= (T\wedge T^n)x. 
	\end{align*}
	As $E_+$ is dense in $F_+$, this shows that $\tilde T \wedge \tilde T^n = 0$ if and only if $T \wedge T^n = 0$.
	Thus, the assertion follows from Proposition \ref{prop:axmann}.
\end{proof}

\section*{Acknowledgements}
The author was supported by the graduate school 
\emph{Mathematical Analysis of Evolution, Information and Complexity} during the work on this article
and he would like to thank Wolfgang Arendt for many helpful discussions 
and the anonymous referee for his/her constructive comments.

\bibliographystyle{abbrv}
\bibliography{../analysis}

\end{document}